\documentclass{amsart}
\usepackage{amsmath,amssymb,enumerate}
\usepackage{epsf,epsfig,amsfonts,graphicx}  
\usepackage[usenames,dvipsnames]{color}
 \usepackage[usenames,dvipsnames]{pstricks}
 \usepackage{epsfig}
 \usepackage{pstricks-add}
 \usepackage{hyperref}
 \usepackage[applemac]{inputenc} 

 \date{August 28, 2014}   
  
 \numberwithin{equation}{section}   
\newcommand{\tq}{\, \big| \, }
\renewcommand{\r}{\mathbb{R}}

\DeclareMathOperator{\vol}{vol}%
\DeclareMathOperator{\Vol}{Vol}%
\DeclareMathOperator{\BL}{BL}%

 \newcommand{\Ffunk}{F_{\text{\tiny \rm Funk}}}
 \newcommand{\RFfunk}{F_{\text{\tiny \rm RFunk}}}
 \newcommand{\Fhilb}{F_{\text{\tiny \rm Hilb}}}

 \newcommand{\dfunk}{d_{\text{\tiny \rm Funk}}}
 \newcommand{\drfunk}{d_{\text{\tiny \rm RFunk}}} 
 \newcommand{\dhilb}{d_{\text{\tiny \rm Hilb}} }

  \newcommand{\gaff}{g_{\text{\tiny \rm Aff}}}
  \newcommand{\gbl}{g_{\text{\tiny \rm BL}}}
  \newcommand{\gjohn}{g_{\text{\tiny \rm John}}}

 \newcommand{\mubl}{\mu_{\text{\tiny \rm BL}}}
  
 \newcommand{\mujohn}{\mu_{\text{\tiny \rm John}}}

\newtheorem{theorem}{\rm\bf Theorem}[section]
\newtheorem{proposition}[theorem]{\rm\bf Proposition}

\newtheorem{corollary}[theorem]{\rm\bf Corollary}

\theoremstyle{definition}
\newtheorem{definition}[theorem]{\rm\bf Definition}
\newtheorem{remark}[theorem]{\rm\bf Remark}
\newtheorem{remarks}[theorem]{\rm\bf Remarks}

\newtheorem{examples}[theorem]{\rm\bf Examples}
\newcommand{\weg}[1]{}

\title{Completeness  and incompleteness of  the  Binet-Legendre Metric}

\author{Vladimir S. Matveev} 
\address{Institut f\"ur Mathematik, Friedrich-Schiller Universit\"at Jena\\
07737 Jena, Germany}  
\email{vladimir.matveev@uni-jena.de}

\author{Marc Troyanov} 
\address{Section de Math{\'e}matiques,  
\'Ecole Polytechnique F{\'e}derale de Lausanne, station 8,
1015 Lausanne - Switzerland} 
\email{marc.troyanov@epfl.ch}

\begin{document}

\bigskip

\begin{abstract} 
The goal of this short paper is to give condition for the completeness of the 
Binet-Legendre metric in Finsler geometry.
The case of the Funk and Hilbert metrics in a convex domain are discussed.
 
\medskip
 
\end{abstract}

\thanks{We thank the Friedrich-Schiller-Universität Jena, EPFL and the Swiss nAtional Science Foundation for their support.}
 
\maketitle

\vfill

\section{Introduction and statement of the main result}

Given a Finsler manifold $(M,F)$ there are several natural ways to construct a Riemannian metric
$g$ on the manifold $M$ that is associated to the given Finsler metric. Recently, such   constructions   
were  shown to be a useful tool in Finsler geometry, see  for example  \cite{berwald,MRTZ,MT,Torrome, Vi}. 

\medskip

Remarkably, in most results of all these papers, 
only the following two  properties of the constructions were  used:
\begin{enumerate}
\item  The construction is pointwise: the associated Riemannian metric $g$, restricted to any tangent space of the manifold $M$ depends only 
on the restriction of the Finsler metric to this tangent spaces. 
\item The construction is homogenous: If we multiply the Finsler metric by a conformal factor $\lambda $, the associated Riemannian metric 
is multiplied by $\lambda^2$. 
\end{enumerate}

In particular, the proofs of most results in the papers  \cite{Vi,berwald,MRTZ,MT}  could be based on any construction of  Riemannian metric satisfying the above two conditions,  at least when smooth and strictly convex Finsler metrics are considered.   
 
 \medskip

The construction in  \cite{MT} is called the \emph{Binet-Legendre metric}\footnote{The construction is slightly older and appeared in \cite{Centore} 
but it usefulness was overseen  until it was reinvented in \cite{MT}}  and has proven to be a flexible and useful tool in Finsler geometry,  its definition will be recalled in subsection \ref{sec.Binet}. 

Our goal in the present paper is to relate the completeness, or incompleteness, of the Binet-Legendre metric to that of the given Finsler metric. Our main result is in fact  the following stronger  Theorem:

\newpage
 
\begin{theorem}[Main Theorem]\label{main:th}
Let $(M,F)$ be a   continuous  Finsler manifold and $\gbl$ be its Binet-Legendre metric, then
there exists a constant $C_1>0$ such that for any $x \in M$ and
 any $\xi \in T_xM$ we have 
 \begin{equation} \label{eq:mainA}  
 \sqrt{\gbl(\xi,\xi)} \leq C_1\cdot F(x,\xi).
 \end{equation} 
 If the Finsler metric $F$ is   quasireversible, then there exists  constants $C_2,C_3>0$  such that  
 \begin{equation} \label{eq:mainAB}  
C_2 \cdot F(x,\xi) \leq \sqrt{\gbl(\xi,\xi)} \leq  C_3\cdot F(x,\xi),
 \end{equation} 
for all $(x, \xi) \in TM$. In particular $\gbl$ and $F$ are Bilipshitz equivalent.
\end{theorem}

\medskip

\begin{remarks} 
$\bullet$ \   Our proof will give explicit (though perhaps not optimal) values for the constants $C_1,C_2,C_3$.
The constants $C_1$ and $C_3$ play the same role, but in the reversible case, we have a better constant (namely 
$C_3 \leq C_1/\sqrt{n}$).
 
$\bullet$ \    Our theorem implies that if the Binet-Legendre metric associated to a Finsler metric $F$ is complete, then 
the Finsler metric is also complete. The converse statement holds in the case of quasi-reversible metric but
not in general. We illustrate this phenomenon by an example in subsection  \ref{sec.counterexample}. 

$\bullet$ \  The quasireversibility hypothesis in the second statement is necessary. For instance the Funk metric (discussed below) 
is forward complete but not backward complete, hence it cannot be bilipschitz equivalent to any Riemannian metric.
In fact it is quite clear from the  Theorem that 
\emph{a Finsler metric is bilipschitz to a Riemannian metric if and only if it is quasi-reversible} (note the ``if'' direction follows from the main Theorem, while the  ``only if'' direction is obvious).   

$\bullet$ \    The other Riemannian metrics constructed in  the papers  \cite{Vi,berwald,MRTZ} involve the second derivatives of the given 
Finlser metric and they generally do not satisfy   (\ref{eq:mainA}).
\end{remarks}

 \medskip
 
The rest of the paper is organized as follows: In section \ref{definitions}  we recall some basic definitions from Finsler geometry and 
we recall the definition and some  basic properties  of the Binet-Legendre metric. 
In section \ref{johnellipse} we discuss another auxiliary Riemannian metric, based on the John ellipsoid from convex geometry, and we use it as  a  tool to prove 
the main   Theorem in section \ref{sec.ProofMain},  where we also derive some of its simple but important consequences.   

In  section  \ref{sec.convexdomain} we discuss some examples. We first recall in  in subsection \ref{Zermelo}  the definition of Zermelo metrics in Euclidean domains and in particular the Funk and reverse Funk metrics. In the next subsection \ref{se.ZBL},  we explicitly compute the Binet-Legendre
metric associated to a Zermelo metric and in subsection \ref{sec.counterexample}  we construct an example of a complete Finsler metric
with incomplete Binet-Legendre metric. In subsection  \ref{sec.Hilbert} we discuss the Hilbert Finsler metric in a convex domain and we use it to compare the Binet-Legendre metric to the so called affine metric, which is another important Riemannian metric defined in an arbitrary convex domain.
The papers ends with an appendix  in which we show by an example  that the Riemannian metric obtained from  the John ellipsoid construction may be nonsmooth, even if the initial Finsler metric is smooth. 

 \medskip
 
\emph{Acknowledgment.}  The  authors are thankful to Rolf Schneider for useful discussions.

\section{A brief review of Finsler Geometry}   \label{definitions}

\subsection{Basic definitions: Finsler manifolds, completeness and quasi-reversibility} 

A \emph{Finsler metric} on a smooth manifold $M$  is defined to be a  continuous 
 function $F:TM\to [0,\infty)$ such that for every point $x\in M$ the  restriction $F_x = F_{|T_xM}$ is a  \emph{Minkowski norm},
 that is,  it satisfies the following properties:
 \begin{enumerate}[i.)]
 	\item  $F(x,\xi) >0$ for any $x\in M$ and any $\xi \neq 0$ in the tangent space $T_xM$,
 	\item $F(x, \xi+\eta) \leq F(x,\xi)+ F(x,\eta)$,
	\item $F(x,\lambda \xi)=\lambda F(x,\xi)$  for all $\lambda\geq 0$,
\end{enumerate}
 for any $x\in M$ and $\xi, \eta \in T_xM$.
 
 \smallskip
 
The Finsler metric is said to be  $c$-\emph{quasireversible},    $1\leq c < \infty$,  if  
 $
   F(x,-\xi)  \leq  c \cdot F(x,\xi)
 $
 for any $(x,\xi) \in TM$. It is called  \emph{reversible} if it is $1$-quasireversible, clearly 
 $F$ is reversible if and only if $F_x$ is a norm in every tangent space. 
 
 \smallskip
 
The \emph{distance} $d(x,y)$ between two points $x$ and $y$ on a Finsler manifold $(M,F)$ is defined to be
the infimum of the length 
$$\ell (\gamma ) = \int_0^1 F(\gamma (t) , \dot \gamma (t)) dt.$$
of all smooth curves  $\gamma:  [0,1] \to M $ joining these two points. 
This distance satisfies the axioms of a metric, except perhaps the symmetry. In fact the condition $d(x,y) = d(y,x)$ 
is satisfied if and only if  the Finsler metric is reversible. 
Together with the distance comes the notion of completeness: the Finsler manifold $(M,F)$ is said to be \emph{forward complete} if every forward Cauchy sequence converges. A sequence $\{ x_i\}\subseteq M$ is \emph{forward Cauchy} if for any $ \varepsilon >0$, there exists an integer $N$ such that $d(x_i,x_{i+k}) <  \varepsilon$ for any $i \geq N$ and $k \geq 0$ (we similarly define \emph{backward Cauchy sequences} by the condition $d(x_{i+k},x_i) <  \varepsilon$, and the corresponding  notion of \emph{backward completeness}).
For  a  quasireversible Finsler metric,  forward completeness is evidently equivalent to backward completeness and will simply be called 
\emph{completeness}.

\medskip

A Finsler manifold is  equipped    with a natural  measure: 
Recall first that   a \emph{density} on the differentiable manifold 
$M$ is a Borel measure $d\nu$ such
that on any coordinate chart $\phi : U \subset M \to \r^n$, the measure $\phi_* d\nu $
is absolutely continuous with respect to the Lebesgue measure, that is it can be written as
\begin{equation}\label{eq.density1}
 \phi_* d\nu =  a(x) dx_1dx_2\dots dx_n
\end{equation}
where  $x_1,x_2\dots ,x_n$ are the coordinates defined by the chart $\phi$ and 
$a(x)$ is a positive measurable function.  
\smallskip

A density on the manifold $M$ naturally induces a Lebesgue measure $d\tau_x$ on (almost) each tangent space $T_xM$, this measure
is given by 
$$
 d\tau_x = a(x)d\xi_1d\xi_2\dots d\xi_n,
$$
where $\xi_1,\xi_2\dots ,\xi_n$  are the natural coordinates on $T_xM$ associated to  $x_1,x_2\dots ,x_n$ and $a(x)$ is given by
 (\ref{eq.density1}).

\smallskip

The \emph{Busemann measure} $d\mu_F$ on the Finsler manifold  $(M^n ,F)$  is then defined to be the unique density  on $M$ such that for every $x\in M$ the volume of  the Finsler unit ball  
$\Omega_x\subset T_xM$ coincides with the volume of the the standard $n$-dimensional  Euclidean unit ball, which we denote by $\omega_n$. It can be calculated from the formula 
$$
  d\mu_F = \frac{\omega_n}{ \nu(\Omega_x)} \, d\nu,
$$
where  $ d\nu $ is an arbitrary continuous density on $M$. It is obvious that the Busemann measure $d\mu_F$ is independent of the chosen density $d\nu$. It is also clear that
in the special case where $F = \sqrt{g}$  for some  Riemannian metric $g$,  the Busemann measure coincides with the Riemannian volume measure, that is    $d\mu_F= d\vol_g$.
 
\smallskip

It is also known, but somewhat delicate to prove, that  if $F$ is a reversible Finsler metric on $M$, then $d\mu_F$ coincides with the $n$-dimensional Hausdorff measure of the metric space associated to the Finsler  structure, see \cite{Al0,B-B-I,Bus1947,Bus1950}.

\subsection{The  Binet-Legendre  metric} \label{sec.Binet}

The Binet-Legendre metric is a canonical Riemannian metric attached to any Finsler metric
on a smooth manifold, it has been invented and  studied in \cite{Centore,MT}.    
Let us recall the construction: Given a Finsler manifold $(M,F)$ and a point $x$ in $M$, we denote by 
$\Omega_x = \{ \xi \in T_xM \tq F(x,\xi) < 1\}$ the  $F$-unit tangent ball at $x$. We then define a scalar
product on the cotangent space $T_x^*(M)$ by
\begin{equation}\label{eq.defgF3}
     \gbl^*(\theta, \varphi) =  
   \frac{(n+2)}{\lambda(\Omega_x)}\int_{\Omega_x} \left( \theta(\eta)\cdot \varphi(\eta)\right) \,  d\lambda (\eta),
\end{equation}
where  $\lambda$  is a Lebesgue measure on  $T_xM$.
Note that this is (up to a constant)  the $L^2$-scalar product of the linear functions $\theta$
 and $\phi$ restricted to $\Omega_x$.

\begin{definition}
The \emph{Binet-Legendre}  metric  $\gbl$   associated to the Finsler metric $F$ is the Riemannian metric dual to the 
the scalar product $\gbl^*$ defined above on $T_x^*(M)$.
\end{definition}
 
The Binet-Legendre metric enjoys a number of important properties, let us state in particular the
following 
\begin{theorem}\label{th.BLproperties1} 
If   $(M,F)$ is a    Finsler manifold and  $\gbl$ is its  associated Binet-Legendre metric, then
 \begin{enumerate}[a)]
  \item If  $F$  is of class $C^k$ on the complement of the zero section of $TM$, then $\gbl$  is  a Riemannian metric of class $C^k$.
  \item If  $\varphi$ is an isometry of $(M,F)$, then it is also an isometry of $(M,\gbl)$.
  \item  If  $F_1$, $F_2$ are two Finsler metrics on $M$ such that \ $\frac{1}{\lambda}\cdot F_1 \leq F_2 \leq\lambda  \cdot F_1$ 
  for some  function $\lambda : M \to \r_+$, then  the corresponding Binet-Legendre metrics satisfy
    $$\frac{1}{\lambda ^{2n}}\cdot {\gbl}_1 \leq  {\gbl}_2 \leq \lambda ^{2n} \cdot {\gbl}_1.$$
 \item If the Finsler metric $F$ is derived from a Riemannian metric $g$, that is $F = \sqrt{g}$, then $\gbl = g$.   
\end{enumerate}
\end{theorem}

\smallskip

We refer to   \cite[Theorem 2.4]{MT} for the first statement, which is in fact proven for the wider class
of \emph{partially smooth} Finsler metrics.  The second statement is  obvious and the third and forth statements
are  proved in   \cite[Proposition 12.1]{MT}.

\section{The  John Metric on a Finsler manifold}  \label{johnellipse}

The proof of  the  Theorem \ref{main:th} will use another auxiliary Riemannian metric,  which we call the \emph{John metric},
that is also associated to a Finsler metric $F$ on the manifold $M$.
To explain this metric, recall that any open bounded convex domain   $\Omega \subset \mathbb{R}^n$ contains a unique
ellipsoid of largest volume \cite{John}. This is called the \emph{John ellipsoid } and we denoted it by $J[\Omega] \subseteq \Omega$. 
A careful study of the uniqueness proof shows that the John ellipsoid  depends continuously on the convex body $\Omega$.   
If $\Omega$ is symmetric with respect to  the origin (that is  $-\Omega = \Omega$), then  $J[\Omega]$ is centered at the origin and  
we have
\begin{equation}\label{johnequality1} 
  J[\Omega] \subseteq \Omega \subseteq \sqrt{n} \cdot J[\Omega], 
\end{equation}
see    \cite{Ball}, \cite[page 214]{Barvinok} or \cite[Section 3.3]{Th}. 

\smallskip

The center of the  John ellipsoid is called  the \emph{John point} of $\Omega$ and denoted by $Q_{\Omega}$,
we then define the  \emph{centered John ellipsoid}  of $\Omega$  as   
$$
  J_0[\Omega] =  J[\Omega]-Q_{\Omega}.
$$

It  was   proved by John in   \cite[Theorem III]{John}, 
that for an arbitrary  open bounded convex set $\Omega \subseteq \r^n$, 
we have 
\begin{equation}\label{johnequality2} 
  (\Omega - Q_\Omega)  \subseteq n \cdot J_0[\Omega],
\end{equation}
see also \cite[page 210]{Barvinok}.
Recall that $\Omega$ contains the origin, thus the above inclusion together with the fact that 
$J_0[\Omega]$ is centrally symmetric implies that
$$
 Q \in n \cdot J_0[\Omega].
$$
 It then follows from (\ref{johnequality2}) that
\begin{equation}\label{johnequality3} 
  \Omega    \subseteq  Q + n \cdot J_0[\Omega]   \subseteq  2 n \cdot J_0[\Omega].
\end{equation}
  
\medskip

The centered John ellipsoid allows us to construct a natural continuous Riemannian metric on 
any Finsler manifold. More precisely we have the following statement. 

\begin{proposition}\label{th.john}
 Any  Finsler manifold  $(M,F)$ carries a well defined Riemannian metric $\gjohn$ of class $C^0$ whose unit ball  at any point $x\in M$  
 is the centered John ellipsoid $J_0[\Omega_x]\subseteq T_xM$ of the Finsler unit ball $\Omega_x  \subseteq  T_xM$. Furthermore  the following inequality hold: 
 \begin{equation} \label{ineq.john1}
 \frac{1}{2{n}} \sqrt{\gjohn(\xi,\xi)} \leq  F(x,\xi) 
\end{equation} 
for any $(x,\xi) \in TM$. 
If the Finsler metric $F$ is reversible, then we have the better estimates
\begin{equation} \label{ineq.john}
 \frac{1}{\sqrt{n}} \sqrt{\gjohn(\xi,\xi)} \leq  F(x,\xi) \leq  \sqrt{\gjohn(\xi,\xi)}.
\end{equation}
In particular a reversible Finsler metric $F$  is bilipschitz equivalent to the Riemannian metric $\gjohn$.
\end{proposition}

This Riemannian metric $\gjohn$  will be called the \emph{John metric} associated to the Finsler metric,
it is a natural construction and appeared in the papers \cite{Planche1,Planche2}.
Note that the John metric may fail to be $C^1$, even if the Finsler metric $F$ is analytic, an example is given in the Appendix.

\begin{proof}
Let $\Omega_x \subset T_xM$ be the Finsler tangent unit ball at $x\in M$, and let us denote by $J_0[\Omega_x]\subseteq T_xM$
the corresponding centered John ellipsoid.
This ellipsoid is the unit ball of a uniquely defined positive symmetric  definite bilinear form  on $T_xM$. 
 By continuity of the John ellipsoid, these bilinear forms give us a 
$C^0$-Riemannian metric $\gjohn$ on $M$, that is naturally associated to the Finsler metric $F$.

The inclusion  (\ref{johnequality3}) gives us $\Omega_x \subset 2n \cdot J_0[\Omega_x]$, which immediately implies 
the inequality (\ref{ineq.john1}).
In the reversible case, $\Omega_x \subseteq T_xM$ is symmetric around the origin and from \eqref{johnequality1}  
we have the inclusion   $J[\Omega_x] \subseteq \Omega_x \subseteq \sqrt{n} \cdot J[\Omega_x]$ which 
are equivalent to  (\ref{ineq.john}).
The  proof of the last assertion is straightforward.
\end{proof}
 
\smallskip

\begin{remark}\label{rem.improveJohnineq}
Arguing as in  \cite{John},  one can  improve the inequality (\ref{johnequality3}) as follows:
\begin{equation}\label{johnequality3improved} 
  \Omega    \subseteq  \sqrt{2 n (n+1)} \cdot J_0[\Omega].
\end{equation}
It follows that the inequality (\ref{ineq.john1}) can also be improved as
\begin{equation}\label{ineq.john1+}
\frac{1}{\sqrt{2n(n+1)}} \sqrt{\gjohn(\xi,\xi)} \leq  F(x,\xi).
\end{equation}
\end{remark}

\medskip

Our next result says that the volume form of the John metric is comparable to the Busemann  measure 
of the Finsler metric $F$. 

\smallskip

\begin{proposition}
Let $(M,F)$ be a Finsler manifold with  Busemann measure $\mu_F$. Then the following inequalities hold:
$$
 d\mu_F \leq d\mujohn\leq n^{n} \cdot  d\mu_F.
$$
where $d\mujohn$ is the Riemannian density of the John metric $\gjohn$ associated to the Finsler metric.
\end{proposition}

\textbf{Proof} 
Choose $d\nu = d\mujohn$ as initial density. Since $J[\Omega_x] \subset \Omega_x \subset T_xM$
for any point $x$ in $M$, we have
\begin{equation}\label{ineq.mm}
   \omega_n = \mujohn(J_0[\Omega_x]) = \mujohn(J[\Omega_x]) \leq \mujohn(\Omega_x).
\end{equation}
 Conversely,  using  (\ref{johnequality2}),  we have
 $$
   \mujohn(\Omega_x)  = \mujohn(\Omega_x - Q_x) 
\leq 
 \mujohn(n \cdot J_0[\Omega_x]) 
  =  n^n \mujohn( J_0[\Omega_x])
 = n^n   \omega_n,
$$
where $Q_x$ is the John point of $\Omega_x$. We just proved the inequalities
$
    \omega_n  \leq   \mujohn(\Omega_x)  \leq    n^n   \omega_n,
$
which are  equivalent to (\ref{ineq.mm}).  \qed
 
 \medskip
 
\begin{remark}
Note that, due to  (\ref{johnequality1}), the second inequality in   (\ref{ineq.mm})  can be improved as follows in the case of a reversible Finsler metric:
\begin{equation}
  \mujohn\leq n^{n/2} \cdot  d\mu_F.
\end{equation}
\end{remark}

\section{Proof of   the main Theorem and some consequences} \label{sec.ProofMain}

We first prove the inequality    (\ref{eq:mainA}).
Recall that by definition the dual of the Binet-Legendre metric $\gbl^*$ is given at any point $x\in M$ by the formula 
$$
 \gbl^*(\theta,\theta) = \frac{(n+2)}{\lambda(\Omega)} \int_\Omega \theta^2(\xi) d\lambda(\xi),
$$
where we denote by $\Omega = \Omega_x$ the $F$-unit ball in $T_xM$.  Since $\Omega \supset J[\Omega] = J_0[\Omega] + Q_{\Omega}$,
we have
\begin{eqnarray*}
 \int_\Omega \theta^2(\xi) d\lambda(\xi) & \geq  &  \int_{(J_0[\Omega] + Q_{\Omega})} \theta^2(\xi) d\lambda(\xi) 
 \\ & = &   \int_{J_0[\Omega]} \theta^2(\xi+Q_{\Omega}) d\lambda(\xi)
 \\ & = &   \int_{J_0[\Omega]} \theta^2(\xi) d\lambda(\xi) + 2 \cdot \theta(Q_{\Omega})   \int_{J_0[\Omega]} \theta(\xi)d\lambda(\xi)
 +  \int_{J_0[\Omega]} \theta^2(Q_{\Omega}) d\lambda(\xi)
 \\ & \geq &    \int_{J_0[\Omega]} \theta^2(\xi) d\lambda(\xi). 
\end{eqnarray*}
The last inequality follows from the fact the ellipsoid $J_0[\Omega]$ is centered at the origin, implying that
$$
 \int_{J_0[\Omega]} \theta(\xi)  d\lambda(\xi)  =  \theta \left(\int_{J_0[\Omega]} \xi  d\lambda(\xi) \right) = 0.
$$ 
On the other hand,  inequality (\ref{johnequality2}) implies that $\lambda (\Omega) \leq n^n \lambda (J_0[\Omega])$
and we obtain 
\begin{equation} \label{eq.pfst1}
  \gbl^*(\theta,\theta) = \frac{(n+2)}{\lambda(\Omega)} \int_\Omega \theta^2(\xi) d\lambda(\xi) \geq 
   \frac{(n+2)}{n^n  \lambda (J_0[\Omega])} \int_{J_0[\Omega]} \theta^2(\xi) d\lambda(\xi). 
\end{equation}
Because the ellipsoid $J_0[\Omega]$ is the unit ball of the John metric, we have
\begin{equation} \label{eq.pfst2}
   \frac{(n+2)}{\lambda (J_0[\Omega])} \int_{J_0[\Omega]} \theta^2(\xi) d\lambda(\xi) = \gjohn^* (\theta,\theta).
\end{equation}
Dualizing the inequalities (\ref{eq.pfst1}) and (\ref{eq.pfst2}), and using  (\ref{ineq.john1})  from Proposition \ref{th.john},
we  obtain  
\begin{equation} \label{eq.pfst2a}
  \sqrt{\gbl (\xi,\xi)} \leq n^{n/2}  \sqrt{\gjohn(\xi,\xi)} \leq 2n^{1+n/2} F(\xi),
\end{equation} 
as desired.

\medskip  
 
We now prove the second part of    Theorem   \ref{main:th}. Assume  first   that the Finsler metric $F$ is reversible,
then the inequalities (\ref{ineq.john})  from Proposition \ref{th.john} implies the following inequalities:
\begin{equation} \label{eq.pfst3}
 \frac{1}{\sqrt{n}} \sqrt{\gjohn} \leq  F\leq  \sqrt{\gjohn}.
\end{equation}
Since $\gjohn$ is Riemannian, it is its own Binet-Legendre metric and we conclude from (\ref{eq.pfst3})  and Property (c) in 
Theorem \ref{th.BLproperties1} that 
\begin{equation} \label{eq.pfst4}
 \frac{1}{n^n} \gjohn \leq \gbl \leq n^n \gjohn.
\end{equation}
From (\ref{eq.pfst3}) and (\ref{eq.pfst4}) we then obtain
\begin{equation} \label{eq.pfst5}
 \frac{1}{n^{n/2}} F \leq \sqrt{\gbl} \leq n^{\frac{n+1}{2}}  F.
\end{equation}

\medskip

Assume now more generally that $F$ is $c$-quasi-reversible, that is   $F(x,-\xi) \leq c \cdot F(x,\xi)$ for any $(x,\xi) \in TM$.
Let us set $F'(x, \xi) = \frac{1}{2} (F(x,\xi)+F(x,-\xi))$, then $F'$ is reversible and satisfies $ \tfrac{2}{1+c}F' \leq F \leq \tfrac{1+c}{2}F$.
If $\gbl'$ is the Binet-Legendre metric for $F'$, we then have from Theorem \ref{th.BLproperties1}(c)  that 
\begin{equation} \label{eq.pfst6}
 {\left(\tfrac{2}{1+c}\right)^n} \cdot \gbl' \leq \gbl \leq  \left(\tfrac{1+c}{2}\right)^n \cdot  \gbl'.
\end{equation}
The inequalities (\ref{eq.pfst5}) applied to the reversible Finsler metric $F'$ say that $ \frac{1}{n^{n/2}} F' \leq \sqrt{\gbl'} \leq n^{\frac{n+1}{2}}  F'$,
combinig this with  (\ref{eq.pfst6})  we finally obtain  
$$
 \sqrt{\gbl} \leq  \left(\tfrac{1+c}{2}\right)^{n/2}  \cdot \sqrt{\gbl'} \leq n^{\frac{n+1}{2}} \left(\tfrac{1+c}{2}\right)^{n/2}  \cdot  F'
 \leq  n^{\frac{n+1}{2}} \left(\tfrac{1+c}{2}\right)^{1+n/2}  \cdot  F.
$$
Similarly
$$
 \sqrt{\gbl} \geq  \left( \tfrac{2}{1+c}\right)^{n/2}  \cdot \sqrt{\gbl'} \geq \frac{1}{n^{n/2}} \left( \tfrac{2}{1+c}\right)^{n/2}  \cdot  F'
 \geq   \frac{1}{n^{n/2}}\left( \tfrac{2}{1+c}\right)^{1+n/2} \cdot F.
$$
We rewrite the  last two inequalities:
$$
 \frac{1}{n^{n/2}}\left( \tfrac{2}{1+c}\right)^{1+n/2} \cdot F  \leq   \sqrt{\gbl}  \leq     n^{\frac{n+1}{2}} \left(\tfrac{1+c}{2}\right)^{1+n/2}  \cdot  F.
$$
The theorem is proved.
\qed

\medskip

\begin{remark}
Disregarding the exact constants, we can summarize the argument for the second statement as follows: let us denote by $\BL[F]$ the Binet-Legendre metric of the Finsler metric $F$ and by 
$\sim$ the bilipschitz equivalence, then 
\begin{align*}
    BL[F| &\sim BL[F'] \sim BL[\gjohn'] =\gjohn' \sim F' \sim F.
\end{align*}
\end{remark}

\smallskip

\begin{remark}    
Using remark  \ref{rem.improveJohnineq}, we obtain a  better estimate for 
the constant $C_1$.  Indeed, using (\ref{ineq.john1+}), the  inequality (\ref{eq.pfst2a}) can be improved to
 \begin{equation} \label{eq.pfst2abis}
  \sqrt{\gbl (\xi,\xi)} \leq  \sqrt{2n(n+1)}\, n^{ n/2} F(\xi). 
\end{equation}  
\end{remark} 

\medskip

Let us now state some simple consequences of the main Theorem:

\begin{corollary} \label{cor.I}
Let  $(M,F)$ be an arbitrary Finsler manifold. If the Binet-Legendre metric 
$\gbl$ is  complete, then $F$ is both forward and backward complete.
\end{corollary}

\begin{proof}
 Let $\{x_j\}$ be a forward Cauchy sequence for the metric $F$, then the first statement from the main Theorem implies
 that  $\{x_j\}$ is a Cauchy sequence for the Riemannian metric $\gbl$, it is therefore a convergent sequence by hypothesis.
 The proof for a backward Cauchy sequence is the same.
\end{proof}

\begin{corollary}\label{cor.II}
 Let $(M,F)$ be a quasireversible  Finsler manifold, then
  \begin{enumerate}[a)]
  \item  The Binet-Legendre metric $\gbl$  is complete if and only if  the given Finsler metric $F$ is complete.
  \item  The Riemannian  volume density  of $\gbl$ is comparable to the Busemann density $d\mu_F$.
  \item Two quasireversible  Finsler manifolds are quasi-isometric if and only if the associated Riemannian manifold 
  with their respective Binet-Legendre metrics are quasi-isometrics.
\end{enumerate}
\end{corollary}

\begin{proof}
The property (a) is an immediate consequence of the Main Theorem since completeness is a property which is stable under  bilipschitz equivalence.

\smallskip 

Property (b) is also a consequence of the Main Theorem:  Let us denote by $d\mubl$ the  Riemannian  volume density  of $\gbl$, then the  Busemann density
is 
$$d\mu_F = \frac{\omega_n}{\mubl(\Omega_x)} d\mubl,$$
where $\Omega_x$ is the unit ball in $T_xM$ for the Finsler metric $F$.
Because $\gbl$ is bilipschitz equivalent to $F$, we have $\frac{1}{k} \cdot B_x \subset \Omega_x \subset k \cdot B_x$ for some constant $k$
where $B_x \subseteq T_xM$ is the unit ball for the metric $\gbl$. It follows at once that
\begin{equation} \label{eq.compvol}
 {\omega_n} k^{-n}  d \mubl  \leq  d\mu_F \leq  {\omega_n} k^n   d\mubl.
\end{equation}

\smallskip 

To prove (c), recall that the Finsler manifolds $(M_1,F_1)$ and  $(M_2,F_2)$ are quasi-isometric if there exists a
map $f : M_1 \to M_2$ and a constant $A$ such that for any $p,q \in M_1$ we have $d_{F_2}(f(p),f(q)) \leq A\cdot (d_{F_1}(p,q)+1)$
and for any $y \in M_2$ there exists $x\in M_1$ with $d_{F_2}(f(x),y)  \leq A$ (see e.g. \cite[\S 8.3]{B-B-I}). It is known that quasi-isometry is an equivalence relation among metric spaces and bilipschitz equivalence is clearly a special case of quasi-isometry. The claim follows thus  also immediately from the main Theorem. 
\end{proof}
 
\begin{remark}  The first inequality in (\ref{eq.compvol}) can be improved: it is known that the Riemannian  volume is in fact  always smaller or equal to the Busemann measure, that is  $d\mubl \leq d\mu_F$ and the equality holds if and only if $F$ is Riemannian. This fact also holds without the reversibility assumption and follows e.g. from   \cite[Theorem 1]{LYZ2000}, see also \cite[Theorem 3.2]{Centore}.  
\end{remark}
 
\section{Examples and Applications}   \label{sec.convexdomain}
 
\subsection{Zermelo Metrics in a domain}  \label{Zermelo}

Let us consider a bounded convex domain $\Omega$  in $\r^n$  and a 
$C^k$-map $u : \mathcal{U} \to \Omega$ where $\mathcal{U} \subset \r^n$ is an arbitrary domain
(one may, but need not, assume that $\mathcal{U}=\Omega$).  
The Finsler metric $F_u$ on $\mathcal{U}$ whose associated tangent unit ball at $x \in \mathcal{U}$ 
is the domain $\Omega$ centered at $u(x)$ is called the \emph{Zermelo metric} associated to the map $u : \mathcal{U} \to \Omega$.
Note that in this definition we use the canonical identification $T_x\mathcal{U} \equiv \r^n$.
The Finslerian unit tangent ball at $x \in \mathcal{U}$ is thus given by
$$
  \Omega_x = \Omega - u(x) =  \{ \xi \in \r^n \mid \xi + u(x)  \in \Omega\}.
$$
The   Finsler  metric $F$ is then given by   
$$
 F_u(x,\xi) = \inf \{t>0 \tq \xi \in t (\Omega+u(x))\} =  \inf \left\{t>0 \tq \left(\frac{\xi}{t} + u(x)\right)\in \Omega \right\}.
$$
Equivalently, for any $\xi \neq 0$:
$$
 F_u(x,\xi) > 0 \quad \text{and} \quad \left(\frac{\xi}{F_u(x,\xi)} + u(x) \right)  \in \partial  \Omega.
$$
We refer to \cite[\S 1.4]{ChernShen} for a discussion of the Zermelo metric and the relation with Zermelo's navigation 
problem\footnote{Note that \cite{ChernShen} have an opposite sign
convention for the vector field $u$}.

\smallskip

\psset{xunit=0.85cm,yunit=0.85cm,algebraic=true,dotstyle=o,dotsize=3pt 0,linewidth=0.8pt,arrowsize=3pt 2,arrowinset=0.25}
\begin{pspicture*}(-16.5,-0.9)(-2,5.3)
 \pscurve[linewidth=1pt](-14.59,3.27)(-12.8,4.34)(-9.94,4.5)(-8.53,2.94)(-8.08,0.65)(-9.13,0)(-11.35,1.58)(-13.39,0.27)(-15.84,1.07)(-15.55,2.34)(-14.59,3.27)
 \rput{-29.05}(-3.33,2.06){\psellipse[linecolor=gray,fillcolor=gray,fillstyle=solid,opacity=0.05](0,0)(1.09,0.91)}
\rput{-29.05}(-10.5,3.05){\psellipse[linecolor=gray,fillcolor=gray,fillstyle=solid,opacity=0.05](0,0)(1.09,0.91)}
 \pscurve[linewidth=1pt]{->}(-10.96,3.48)(-8.6,3.96)(-6.8,3.9)(-5.03,3.33)(-3.85,2.5) 
\psdots[dotstyle=*](-10.96,3.48)
\psdots[dotstyle=*](-3.78,2.46)
\rput[bl](-6.8,4.015){$u$}
\rput[bl](-11.28,3.48){$x$}
\rput[bl](-3.8,2.49){$u(x)$}
\rput[bl](-4.16,0.9 ){$\Omega$}
\rput[bl](-11.9,2.5){$\Omega_x$}
\rput[bl](-13.31,-0.1){$\mathcal{U}$}
\end{pspicture*}

\centerline{{\begin{minipage}{10cm} \small
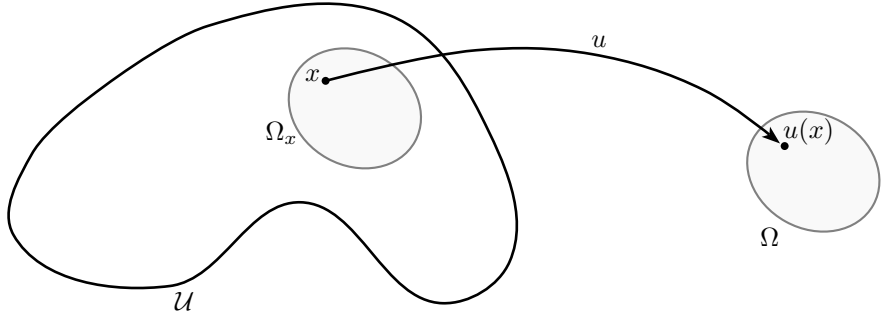
 Figure 1. A Zermelo metric in the domain $ \mathcal{U}$, the Finlser unit ball at $x$ 
is given by $\Omega$ with $u(x)$ as origin. 
\end{minipage}}}

\begin{examples}
\ \textbf{(a)}  If $u(x) = c \in \Omega$ is constant, then the corresponding Zermelo    metric is is invariant by translation, it is thus the  Minkowski metric whose unit ball
is given by $\Omega -c$.
\\  \textbf{(b)} If $\mathcal{U} = \Omega$ and $u(x) = x$ is the identity map, then the corresponding Zermelo metric is called the 
\emph{Funk metric}  and denoted by $\Ffunk$. The Finsler unit ball at the point $x\in \Omega$ is the convex domain $\Omega$ itself, but with the point $x$ as its
center (this metric  is therefore also called the \emph{tautological Finsler structure}).  
\\  \textbf{(c)} The reverse of a Zermelo metric is also a Zermelo metric. Recall that the reverse of a Finsler metric $F$ is the Finsler metric
${}^rF$ given by ${}^rF(x,\xi) = F(x,-\xi)$. In the case of the Zermelo metric $F_u$ associated to the map $u : \mathcal{U} \to \Omega$, we easily check that the
reverse metric ${}^rF_u$ is the Zermelo metric  associated to the map $-u : \mathcal{U} \to -\Omega$, that is we have the identity
$$
  {}^rF_u(x,\xi) = F_{-u}(x,\xi) = F_u(x, -\xi).
$$
\textbf{(d)} In particular the reverse of the Funk metric, which is denoted by $\RFfunk$, is the Zermelo metric in $\mathcal{U}=\Omega$
associated to the map $u : \Omega \to - \Omega$ given by $u(x) = -x$.  The Finsler unit ball is the symmetric 
image of $\Omega$ with respect to the center of symmetry at  $x$.
\end{examples}


We refer to \cite{ChernShen} and Chapters 2 and 3 in \cite{Handbook}  for some background on Funk and
reverse Funk geometry. In particular, the following formula  for the distance is well known: if   
$p$ and $q$ are distinct  points in $\Omega$ and $a,b$  are the two points lying on the intersection of the line through $p$ and $q$ 
with  the boundary $\partial \Omega$,
and if  $a,q,p,b$  appear in that order on that line then 
\begin{equation} \label{eq.dfunk}
    \dfunk(p,q) = \log \left(\frac{|a-p|}{|a-q|}\right) \quad  \text{and}  \quad
     \drfunk(p,q) =   \log \left(\frac{|b-q|}{|b-p|}\right).
\end{equation}

\begin{pspicture*}(-7,-3.0)(5.62,2.8)
\psset{xunit=0.7cm,yunit=0.7cm,algebraic=true,dotstyle=*,dotsize=3pt 0,linewidth=0.8pt,arrowsize=3pt 2,arrowinset=0.25}
\pscurve[linewidth=1.2pt](-2,2)(0.33,2.95)(1.98,2.36)(3.09,0.78)(2,-2)(-1.2,-2.9)(-3,-2)
\psline[linewidth=1.2pt](-2,2)(-3.16,0.53)
\psline[linewidth=1.2pt](-3.16,0.53)(-3,-2)
\psline[linestyle=dashed,dash=1pt 1pt](-1.2,-2.9)(1.98,2.36)
\rput(3,-1.43){$\partial \Omega$}
\psdots[dotsize=0.12](0,-0.92)
\rput(0.32,-0.94){$p$} 
\psdots[dotsize=0.12](0.99,0.72)
\rput(1.28,0.68){$q$} 
\psdots[dotsize=0.12](-1.2,-2.9)
\rput(-1.4,-3.15){$b$} 
\psdots[dotsize=0.12](1.98,2.38)
\rput(2.2,2.6){$a$} 
\end{pspicture*}   

\centerline{{\begin{minipage}{10cm} \small
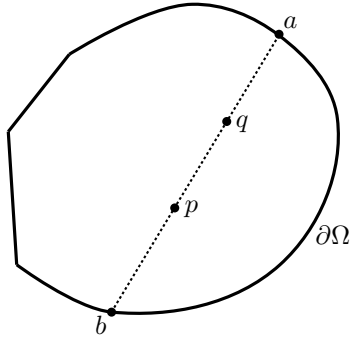
 Figure 2. The distance in the Funk or reverse Funk metric is given by the logarithm of the ratio 
 of the Euclidean distances to the boundary. 
\end{minipage}}}
 
\bigskip

The following facts are classical:
\begin{proposition}  Let   $\Omega\subset \r^n$ be a bounded convex domain, then 
 \begin{enumerate}
 \item The reverse Funk metric satisfies $\RFfunk(x,\xi) = \Ffunk(x,-\xi)$. They are both 
  invariant under affine transformations preserving  $\Omega$.
  \item The Funk metric in  $\Omega$  is forward complete but not backward complete.
  The reverse Funk metric is backward complete and not forward complete.  
  \item  Both  metrics are projective, meaning that the Euclidean straight lines are geodesics.
 \item  If the bounded convex domain $\Omega\subset \r^n$ has a boundary of class $C^k$, then 
  $\Ffunk$ and $\RFfunk$ are also of class $C^k$ (on the complement of the zero section). 
\end{enumerate}
\end{proposition}
The  first statement is obvious and statement 2 and 3 are proved in   \cite{ChernShen} 
for  domains with smooth and strongly convex domains and in Chapters 2 and 3 in \cite{Handbook}.
The last statement follows from the implicit function theorem.

\smallskip 
 
\subsection{Computation of the Binet-Legendre metric for a Zermelo metric}\label{se.ZBL}

Since the Funk metric is not backward complete,  it follows from Corollary  \ref{cor.I} that its associated 
Binet-Legendre metric is  incomplete. In this section we provide another proof for the incompleteness. More 
generally we  compute the Binet-Legendre metric for a general
Zermelo metric in a domain $\mathcal{U}$ and show that it is never complete unless $\mathcal{U}= \r^n$.

\smallskip

We will in fact consider a more general situation.  A  Borel probability $\mu$  measure on $\r^n$ 
is said to have \emph{finite quadratic moment} if 
\begin{equation}\label{cond.mu1}
  \int_{\r^n} |x|^2d\mu (x) < \infty,
\end{equation}
where $|x|$ is the Euclidean norm. The probability $\mu$ is said to be \emph{affinely non degenerate} 
if for any non zero linear form $\varphi : \r^n \to \r$ and any point $a$ in $\r^n$ we have
\begin{equation}\label{cond.mu2}
  \int_{\r^n} \varphi^2(x-a)d\mu (x) > 0.
\end{equation}
Equivalently the support of the measure $\mu$ is not contained in an affine hyperplane.
Given such a measure  $\mu$ satisfying satisfying (\ref{cond.mu1}) and (\ref{cond.mu2}) we associate to any 
$C^k$  smooth function $u : \mathcal{U} \to \r^n$  the following scalar product on $(\r^n)^*$:
$$
 g^*_x(\theta,\varphi) = \gamma \int_{\r^n}\theta(\zeta - u(x)) \cdot  \varphi(\zeta - u(x))\,  d\mu (\zeta),
$$
where the constant $\gamma >0$ is an arbitrary parameter. 
If the measure is centered at the origin, that is $\int d\mu = 0$, then we have    
 \begin{eqnarray*} 
 g^*_x(\theta, \varphi) 
 & = & \gamma \left\{  \int_{\mathbb{R}^n}  \theta(\zeta) \varphi(\zeta) d\mu (\zeta)  
 -  \theta(u(x))   \int_{\mathbb{R}^n}   \varphi(\zeta) d\mu (\zeta) \right.
 \\ & & \hspace{2.4cm} -  \left. \varphi(u(x))  \int_{\mathbb{R}^n}  \theta(\zeta) d\mu (\zeta)  +  \theta(u(x))   \varphi(u(x))  
  \right\}
\\  & = &  \gamma \int_{\r^n}\theta(\zeta) \varphi(\zeta) d\mu (\zeta)
  + \gamma \cdot \theta(u(x)) \varphi(u(x))
\\  & = &  g^*_0(\theta,\varphi)  
  + \gamma \cdot \theta(u(x)) \varphi(u(x)).
\end{eqnarray*} 
In the second  equality we have used $ \int    \varphi(\zeta) d\mu (\zeta)
= \varphi\left(\int \zeta  d\mu (\zeta)\right) = 0$.
If we furthermore assume that the   coordinates
are chosen to be orthonormal for the   metric $g_0$ at the origin, then the coefficients
of $g^*_x$ are
\begin{equation} \label{moment.dualmetric}
 g^{ij} = g^*_x(\varepsilon^i, \varepsilon^j) = \delta^{ij} + \gamma u_iu_j,
\end{equation}
where $\varepsilon^1, \dots, \varepsilon^n$ is the dual canonical basis.
Inverting this matrix, one obtains the following Riemannian metric 
on $\r^n$:
\begin{equation} \label{moment.metric}
 g_{ij} = \delta_{ij} - \frac{\gamma \,   u_iu_j}{1+\gamma |u(x)|^2}. 
\end{equation}

\medskip

The Binet-Legendre metric for the Zermelo metric corresponding to the function $u : \mathcal{U} \to \Omega$, 
where $\Omega$ is a bounded convex domain,  is the special case of this construction 
corresponding to the constant $\gamma = (n-2)$ and the measure $d\mu =  
\frac{1}{\Vol (\Omega)}\chi_{\Omega}\,dx$. Indeed, by definition of the Zermelo metric, the Finsler tangent ball at  a point $x\in \mathcal{U}$ is given by
$$
 \Omega_x = \{\xi \in T_x\mathcal{U} \mid  F (x,\xi) < 1\} = \{\xi \in \r^n \mid \xi \in (\Omega - u(x) )\} =  \Omega - u(x),
$$ 
(here we use the canonical identification $T_x\mathcal{U} = \r^n$).   The dual Binet-Legendre metric associated to the Zermelo metric  is then given by
\begin{eqnarray*}
 g^*_x(\theta, \varphi) & = & \frac{(n+2)}{\Vol(\Omega_x)} \int_{\Omega_x} \theta(\xi) \varphi(\xi) d\xi
=
  \frac{(n+2)}{\Vol(\Omega_x)}\int_{\Omega} \theta(\zeta-u(x)) \varphi(\zeta-u(x)) d\zeta.
\end{eqnarray*} 
It follows that  in an appropriate coordinate system, the Binet-Legendre associated to 
a Zermelo metric in a domain  $\mathcal{U}$ is given by $(\ref{moment.metric})$ with $\gamma = (n+2)$.

\smallskip

Observe in particular that since $u(x)$ belongs to the bounded domain $\Omega$ for any $x \in \mathcal{U}$,
the tensors  $(\ref{moment.metric})$  and $(\ref{moment.dualmetric})$ are always bounded. This implies in particular that the 
Binet-Legendre metric of a Zermelo metric in a domain  $\mathcal{U}$ is bilipschitz equivalent to the Euclidean metric.
In particular it is complete if and only if    $\mathcal{U} = \r^n$. 

\begin{remark}
 In the special case of the Funk or reverse Funk metric, we have $u(x) = \pm x$. 
 It follows from  (\ref{moment.metric})  that for any bounded convex domain $\Omega \subset \r^n$, 
 the Binet-Legendre metric is given in some 
 coordinate system  by
 \begin{equation*} \label{moment.Funkmetric}
 g_{ij} = \delta_{ij} - \frac{\gamma \,   x_ix_j}{1+\gamma |x|^2}. 
\end{equation*}
Observe that this formula is independent of the geometry of $\Omega$.
\end{remark}

\begin{remark} 
 The previous construction of a metric associated to a probability measure in $\r^n$ 
 is natural in  multivariate statistics. Let  $X_1, \dots , X_n$ be random variables
 and assume that the random vector  $X = (X_1, \dots , X_n)$ is non degenerate.
 Recall that this means that there are no constants $a_1, \dots, a_n$ and $C$ such
 that $ \mbox{Prob}(\sum_i a_iX_i=C) = 1$. Assume also  that the 
 random vector $X$  has finite second moments, that is $\mathbb{E}(X_i^2)<\infty$ ($1\leq i \leq n$),
 where $\mathbb{E}(\,  )$ is the expectation.
The \emph{joint distribution} of those variables is the probability measure $\mu$ on $\r^n$ defined by
$\mu(B) = \mbox{Prob}(X\in B)$   for any Borel set $B \subset \r^n$,  and under the given hypothesis
the measure $\mu$ satisfies the previous conditions  (\ref{cond.mu1}) and (\ref{cond.mu2}).
Choosing the function $u(x) = x$, the corresponding metric  $g_x^{ij}$   in $(\r^n)^*$ is then the 
matrix of \emph{product moments}:
$$
 g_x^{ij} =  \mathbb{E}((X_i-x_i)(X_j-x_j)).
$$
At the barycenter of $\mu$, this matrix is the \emph{covariance matrix} of the random vector $X$ and is often denoted by $\Sigma$. 
The inverse matrix $g_{ij}$ is called the \emph{precision} or \emph{concentration} matrix. In the case of Gaussian random variables, 
this matrix  is related to  conditional independencies between the random variables.
\end{remark}

 \smallskip 
 
\subsection{An example of a complete metric with incomplete Binet-Legendre metric} \label{sec.counterexample}

In this subsection we briefly give an example of a Finsler metric that is both forward and  backward complete
and whose associated Binet-Legendre metric is incomplete, showing that the converse to Corollary \ref{cor.I}  fails.

\medskip

The example is given by a Zermelo metric that interpolates between the Funk metric (which is forward complete) and 
the reverse Funk metric (which is backward complete). It can be built in any bounded convex domain, but we will only 
describe it in the standard unit ball $\mathbb{B}^n \subset \r^n$. 

\medskip

Using (\ref{eq.dfunk}), we see that  the  Funk distance in $\mathbb{B}^n$ from the origin to a point $x\in \mathbb{B}^n$ is given by
$$
  \dfunk(0,x) = \log\left(\frac{1}{1-|x|} \right),
$$
therefore the open ball of radius $t$ centered at the origin for the Funk metric is given by
$$
  W_t = \{ x \in \r^n \mid |x| < 1-\mathrm{e}^{-t}\}.
$$
Let us now choose a smooth  function $u : \mathbb{B}^n \to \mathbb{B}^n$ such that for any integer $k \in \mathbb{N}$
we have
$$
 u(x) =  \begin{cases}
    \  x  & \text{ if } \,   x\in W_{4k+1} \setminus W_{4k}, \\
   -x  & \text{ if } \, x\in W_{4k+3} \setminus W_{4k+2}.
\end{cases}
$$
The Zermelo metric $F_u$ associated to the function $u$, coincides with the Funk metric
in $W_{4k+1} \setminus W_{4k}$ and to the reverse Funk metric in 
$W_{4k+3} \setminus W_{4k+2}$ for any integer $k \in \mathbb{N}$.
In particular we have 
$$
 x \not\in W_{4k+3} \quad  \Longrightarrow \quad d_u(0,x) \geq k  \text{ and } d_u(x,0) \geq k. 
$$
Because  $W_{4k+3}$ is relatively compact in $\mathbb{B}^n$, it is 
clear that the metric $F_u$ is both forward and backward complete. 
Since we proved in the previous subsection that the associated  Binet-Legendre 
metric is not complete, we have produced an example of a complete Finsler metric
with incomplete Binet-Legendre  metric.

  \medskip
  
\subsection{The  Hilbert metric and the  ``affine metric'' in a bounded convex domain}\label{sec.Hilbert}

The symmetrization of the Funk metric in a bounded convex domain $\Omega$ is called
the\emph{ Hilbert metric} in that domain, the Finsler norm is thus given by
\begin{equation} 
    \Fhilb(x, \xi) = \frac{1}{2}(\Ffunk(p,\xi) + \Ffunk(p,-\xi)).
\end{equation}
Referring to the notations in Figure 2, we have the
following formula for the   distance between two points $p$ and $q$:
\begin{equation} \label{eq.dhilb}
\dhilb(p,q) =    \frac{1}{2} \left( \dfunk(p,q)+ \dfunk(q,p) \right)
     =    \frac{1}{2}\log \left(\frac{|a-p|}{|a-q|} \cdot \frac{|b-q|}{|b-p|}\right),
\end{equation}
 We  refer to the books \cite{Busemann1955,Papadopoulos} for a short introduction to Hilbert Geometry 
 and to \cite{Handbook} for  an overview of some recent developments.  

\medskip

We have the following result about the Binet-Legendre metric associated to the Hilbert metric:

\begin{proposition}  
The Binet-Legendre Metric associated to the Hilbert metric in a  bounded convex domain $\Omega$ is a complete Riemannian metric,
that is it is invariant under the group of  projective transformations preserving the domain. It is bilipschitz equivalent to the Hilbert metric.
\end{proposition}

\begin{proof}
The  Hilbert metric is clearly reversible and it  is not difficult to check from the formula (\ref{eq.dhilb}) that 
it is  complete. The second statement in Theorem \ref{main:th} implies that its associated Binet-Legendre 
metric is  bilipshitz equivalent to the Hilbert metric, in particular it is also  complete.

The Hilbert metric is   invariant under projective transformations since the distance is expressed in terms of 
the cross ratio of four aligned points. Using statement (b) from Theorem \ref{th.BLproperties1}, we deduce
that the Binet-Legendre metric is also   invariant under projective transformations
\end{proof}

\bigskip

 Another  important projectively invariant metric in a convex domain can be constructed from 
 the solution to some Monge-Ampère equation. It is 
 based on the following 
\begin{theorem}
 Let $\mathcal{U} \subset \r^n$ be an arbitrary  bounded convex domain. Then there exists a unique 
 solution to the   following Monge-Ampère equation:
\begin{equation} \label{eq.MongeAmpere}
   \det \left(\frac{\partial ^2 u}{\partial x_i \partial x_j} \right) = \left(-\frac{1}{u}\right)^{n+2},  
\end{equation}
which is smooth, positive, strictly concave,  continuous in the closure $\overline{\mathcal{U}}$ and vanishes on the
boundary $\partial\mathcal{U}$.
\end{theorem}

This theorem was first proved in 1974 by  C.  Loewner and L. Nirenberg    for the case of smooth, 2-dimensional 
strictly convex  domain \cite{LN74} and in 1977 by S.Y. Cheng and S.T. Yau for the general case \cite{ChengYau1977,LN74,Loftin2011}.

\begin{definition}
 The \emph{affine metric} on the convex domain $\mathcal{U}$ is the Riemannian metric defined as
 $$
   g_{\mathrm{Aff}} = -\frac{1}{u}\sum_{i,j}  \frac{\partial ^2 u}{\partial x_i \partial x_j} dx_i dx_j,
 $$
 where  $u : \mathcal{U} \to  \r^n$ is the above solution to (\ref{eq.MongeAmpere}).
\end{definition}

 Observe that by the  strict concavity of $u$, the metric $g_{\mathrm{Aff}}$ is  positive definite, hence Riemannian.
 The name ``affine metric'' has been proposed in relation to the Blaschke theory of affine hypersurfaces,
 see \cite{BH,Loftin2010,NS}. The affine metric enjoys the following properties:

\begin{theorem}
\begin{enumerate}[i.)]
  \item The affine metric $g_{\mathrm{Aff}}$ is complete and invariant
 under projective transformations leaving the domain $\mathcal{U}$ invariant. 
  \item The affine metric $g_{\mathrm{Aff}} $ is bilipschitz equivalent to the Hilbert metric: there exists a constant $C$
  such that 
  $$
     \frac{1}{c} \Fhilb  \leq  \sqrt{\gaff} \leq  c\cdot  \Fhilb.
  $$
\end{enumerate}
\end{theorem}

A proof of the first statement is given in \cite[sec. 6 and 9]{LN74}, see also \cite{ChengYau1980, ChengYau1986}. The second statement is a recent result
by Y. Benoist and D. Hulin \cite[Proposition 3.4]{BH}.
Observe that the completeness of  $g_{\mathrm{Aff}} $ also follows from the second statement, since
the Hilbert metric is complete.   We then have the following 
\begin{corollary}
 The Binet-Legendre metric $\gbl$ associated to the Hilbert metric  in a properly convex domain $\mathcal{U} \subset  \mathbb{RP}^n$ is 
  bilipschitz equivalent to the affine metric $\gaff$.
\end{corollary}
\begin{proof}
 The corollary follows at once from the previous theorem and the main Theorem \ref{main:th}.
\end{proof}

 \medskip
 
In conclusion, both the Binet-Legendre and the affine metric in a convex domain are complete, invariant under projective 
transformation and bilipschitz  equivalent  to the Hilbert metric. Observe however that the construction of the affine
metric is based on hard analysis to solve  a non--linear elliptic partial differential equation, so even  the existence 
of such a metric is a nontrivial fact.
On the other hand the Binet-Legendre metric is based on a direct and  quite elementary geometric construction.
This metric can be effectively computed, at least for  sufficiently simple  domains see e.g. \cite{Metzner}.
     
 \appendix

\section{Non smoothness of the John metric} 
  
The John metric, like the Binet-Legendre metric, is a natural construction in Finsler geometry 
that enjoys good functorial properties, in particular properties (b), (c)  and (d) of Theorem \ref{th.BLproperties1}
also hold for the John metric. However, the John metric is in general not smooth and this fact creates serious
limits to its potential usefulness in Finsler geometry. We illustrate this phenomenon by the following example: 
Consider the following Finsler metric $F$ on $M=\r^n$:
$$
  F(x,\xi) = \| \xi \|_{p(x)} = \left(\sum_{i=1}^n |\xi_i|^{p(x)}\right)^{1/p(x)},
$$
where $p$ is the function $p(x) = 1+e^{x_1}$. If one identifies $T_xM$ with $\r^n$, the Finsler unit ball is 
$$
 \Omega_x = \left\{\xi \in \mathbb{R}^n \tq \sum_{i=1}^n |\xi_i|^{p(x)} < 1\right\}, 
$$
It is easy to see that the  John ellipsoid of $\Omega_x$ is an  euclidean ball centered at the origin. 
Indeed, each $\Omega_x$ is invariant with respect to the symmetries $\sigma_i : (\dots,\xi_i,\dots) \mapsto (\dots,-\xi_i,\dots)$   \label{page.sigmai}  
and  $\sigma_{ij} : (\dots,\xi_i,\dots, \xi_j, \dots) \mapsto  (\dots,\xi_j\dots, \xi_i, \dots) $,  
and since  the John ellipsoid  $J[\Omega_x]$ of $\Omega_x$  is unique, it must be $\sigma_i$- and $\sigma_{ij}$-invariant 
for all $i,j = 1, \dots, n$. Since  the Euclidean balls centered at the origin are the only ellipsoid invariant with respect to all such symmetries, the John ellipsoid
must be such a ball.  The radius $r$ of  the ball $J[\Omega_x]$  only depends on $p = p(x)$ and a calculation shows that
\begin{equation*} \label{rp}
r(x)= \min\left\{1, n^{\tfrac{1}{2}- \tfrac{1}{p}}\right\} = 
\left\{ \begin{array}{ll}  
n^{\tfrac{1}{2}- \tfrac{1}{p}} &\text{if } 1< p\le 2, 
\\ 1 &\text{if }   p\ge 2.
\end{array} \right. 
\end{equation*}
Indeed, for $p\le 2$ a  common point of the boundary of  $J[\Omega_x]$ of $\Omega_x$  is  given by $ \xi  = (1,0,\cdots,0)$, while for  $p\ge 2$ a  common point
of the boundary of $J[\Omega_x]$ of $\Omega_x$  is   $\xi = \left(\left(\tfrac{1}{n}\right)^{\tfrac{1}{p}},\cdots,\left(\tfrac{1}{n}\right)^{\tfrac{1}{p}}\right)$.

\begin{pspicture*}(0,-2.)(15,1.7)
\psset{xunit=1.0cm,yunit=1.0cm,algebraic=true,dotstyle=o,dotsize=3pt 0,linewidth=0.8pt,arrowsize=3pt 2,arrowinset=0.25}
\psplot[plotpoints=200]{4.00000002928613}{5.999998253827496}{(1-abs(x-5)^(1.5))^(1/1.5)}
\psplot[plotpoints=200]{4.00000002928613}{5.999998253827496}{-(1-abs(x-5)^(1.5))^(1/1.5)}
\psplot[plotpoints=200]{6.500000024554854}{8.499997995302449}{(1-abs(x-7.5)^(2.99))^(1/2.99)}
\psplot[plotpoints=200]{6.500000024554854}{8.499997995302449}{-(1-abs(x-7.5)^(2.99))^(1/2.99)}
\psplot[plotpoints=200]{9}{11}{(1-abs(x-10)^(5.99))^(1/5.99)}
\psplot[plotpoints=200]{9}{11}{-(1-abs(x-10)^(5.99))^(1/5.99)}
\psplot[plotpoints=200]{1.500000034017478}{3.5}{-(1-abs(x-2.5)^(1.2))^(1/1.2)}
\psplot[plotpoints=200]{1.500000034017478}{3.5}{(1-abs(x-2.5)^(1.2))^(1/1.2)}
\pscircle[linestyle=dashed,dash=1pt 1pt](5,0){0.89}
\pscircle[linestyle=dashed,dash=1pt 1pt](10,0){1}
\pscircle[linestyle=dashed,dash=1pt 1pt](7.5,0){1}
\pscircle[linestyle=dashed,dash=1pt 1pt](2.5,0){0.79}
\rput[bl](4.6,-1.55){{$p=1.5$}}
\rput[bl](2.15,-1.55){{$p=1.2$}}
\rput[bl](9.7,-1.55){{$p=6$}}
\rput[bl](7.3,-1.55){{$p=3$}}
\end{pspicture*}

\centerline{{\begin{minipage}{8cm} \small
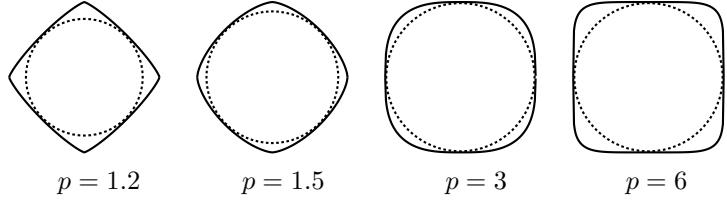
  Figure 3. The convex bodies $\Omega_x$ and their John  ellipsoids. 
\end{minipage}}}
 
\medskip

It is elementary to check that the function $r(x)$ is  not differentiable when $x_1 = \log (2)$, that is 
$p=2$. Therefore the ellipsoid $J[\Omega_x]$ does not depend smoothly on $x$.
Moreover, the metric  $\gjohn$ has a discontinuous curvature and therefore cannot be made smooth by a
$C^0$-change  of  coordinates. 
We have thus constructed an analytical Finsler metric $F$ on $\r^n$ such that the associated John metric is given 
at the point $x$ by
$$
  \gjohn (\xi , \eta) = \frac{1}{r(x)^2}\langle \xi , \eta \rangle,
$$
where $r(x)$ is not differentiable.

\medskip


\end{document}